\documentclass[pdftex,reqno]{amsart}
\usepackage{amssymb,amsthm,amsmath}
\usepackage{cases}
\usepackage{graphicx,xcolor}
\usepackage{comment}

\theoremstyle{plain}
\newtheorem{theorem}{Theorem}[section]
\newtheorem{lemma}[theorem]{Lemma}
\newtheorem{corollary}[theorem]{Corollary}

\theoremstyle{definition}
\newtheorem{definition}[theorem]{Definition}

\newtheorem{remark}[theorem]{Remark}
\newtheorem{assumption}[theorem]{Assumption}

\numberwithin{equation}{section}

\title{Convergence of Levenberg--Marquardt method in Banach spaces}
\author{Akari Ishida and Sei Nagayasu}

\begin{document} 
\begin{abstract}
We study the Levenberg--Marquardt method for nonlinear inverse problems in Banach spaces under a H\"{o}lder stability estimate. Under suitable assumptions on the geometry of the Banach space and the Fr\'{e}chet derivative of the measurement operator, we establish the well-definedness of the method, local convergence and convergence rates for exact data, and the monotonic decrease of the Bregman distance up to the stopping index for noisy data. These results extend the corresponding Hilbert space theory to the Banach space setting. 
\end{abstract} \maketitle
\section{Introduction}
We consider nonlinear inverse problems of the form
\begin{equation} \label{NOE}
F(x)=y,\quad x\in\mathcal{D}(F),\, y\in Y,
\end{equation} 
with the domain $\mathcal{D}(F)\subset X$. Here, $X$ and $Y$ are Banach spaces. Determining the unknown parameter~$x$ from the measurement data~$y$ is generally ill-posed, and regularization methods are used to obtain stable approximations. In this paper, we employ the Levenberg--Marquardt method (LM-method) as an iterative regularization method and study its convergence for a class of inverse problems satisfying a H\"{o}lder stability estimate. One of our main results is that, under our assumptions, the sequence generated by the LM-method is well-defined. Since nonlinear operator equations may have multiple solutions, convergence generally depends on the initial guess and is therefore local. 
For exact data, we prove local convergence and derive convergence rates. For noisy data, we establish the monotonic decrease of the Bregman distance up to the stopping index and finite termination of Morozov's discrepancy principle. The precise statements are given in Theorems~\ref{thm:exact} and \ref{thm:noisy}, respectively.

Extensive research has been carried out to study the convergence of iterative regularization methods for nonlinear inverse problems. The tangential cone condition~(TCC) is a standard assumption for their convergence analysis, see Kaltenbacher--Neubauer--Scherzer~\cite{KNS}, Schuster--Kaltenbacher--Hofmann--Kazimierski~\cite{SKHK} and the references therein. While the TCC is used to establish convergence, additional assumptions, such as source conditions, are generally required for convergence rates. Convergence of the Landweber iteration was established by Hanke--Neubauer--Scherzer~\cite{HNS} in Hilbert spaces, and extended to Banach spaces by Kaltenbacher--Sch\"{o}pfer--Schuster~\cite{KSS}. Local convergence of the LM-method under the TCC was established by Hanke~\cite{Hanke} in Hilbert spaces and by Jin~\cite{Jin} in Banach spaces. 

An alternative approach uses a H\"{o}lder stability estimate. For the Landweber iteration, de Hoop--Qiu--Scherzer~\cite{dHQS} and Mittal--Giri~\cite{MG} established convergence and convergence rates under a H\"{o}lder stability estimate. Corresponding results for the LM-method in Hilbert spaces were obtained in Ishida--Nagayasu--Nakamura~\cite{INN}. A H\"{o}lder stability estimate thus allows convergence and convergence rates to be obtained under a single condition. Moreover, together with the Lipschitz continuity of the Fr\'{e}chet derivative~$F'$, the H\"{o}lder stability estimate implies the TCC. In this paper, we extend the analysis of \cite{INN} from Hilbert spaces to Banach spaces. The use of a H\"{o}lder stability estimate is also motivated by the fact that stability estimates can arise naturally in the study of inverse problems. Indeed, although inverse problems are generally ill-posed, a H\"{o}lder stability estimate can sometimes be obtained by imposing suitable a priori assumptions on the unknown parameters. 

This paper is organized as follows. Section~\ref{section:preliminaries} recalls some basic notions and properties concerning the geometry of Banach spaces and Bregman distances. Section~\ref{sec:preliminaries;LM} introduces the LM-method and establishes its basic properties. Sections~\ref{section:exact} and~\ref{section:noisy} present the convergence analysis for exact and noisy data, respectively.

\section{Geometry of Banach spaces} \label{section:preliminaries}
In this section, we collect some basic notions and results on Banach spaces.
\subsection{Duality mappings} \label{Preliminaries;DM}
We begin by introducing the duality mapping.
\begin{definition}[duality mapping] 
Let $X$ be a Banach space with dual $X^{\ast}$. The set-valued mapping~$J_{p}: X\rightarrow 2^{X^{\ast}}$ with $p\geq 1$ defined by
\[
J_{p}(x)=\left\{x^{\ast}\in X^{\ast}\mid \langle x^{\ast},x\rangle_{X^{\ast}, X}=\|x\|_{X}\|x^{\ast}\|_{X^{\ast}},\|x^{\ast}\|_{X^{\ast}}=\|x\|_{X}^{p-1}\right\}
\]
is called the duality mapping of $X$ with gauge function $t\mapsto t^{p-1}$.
\end{definition}

In general, the duality mapping is set-valued. We introduce the notions of $p$-convexity and smoothness of Banach spaces needed in our analysis.
\begin{definition}[\textup{\cite[\S2.3.1]{SKHK}}] \label{def:convex;smooth}
A Banach space $X$ is said to be
\begin{enumerate} \renewcommand{\labelenumi}{(\alph{enumi})} \setlength{\leftskip}{-0.15in}
\item $p$-convex if there exists a constant $C_{p}>0$ such that
\[
\frac{1}{p}\|x-\widetilde{x}\|_{X}^{p}\geq\frac{1}{p}\|x\|_{X}^{p}-\langle j_{p}(x), \widetilde{x}\rangle_{X^{\ast}, X}+\frac{C_{p}}{p}\|\widetilde{x}\|_{X}^{p}
\]
for all $x, \widetilde{x}\in X$ and all $j_{p} (x) \in J_{p} (x)$;
\item uniformly convex if $\delta_{X} ( \epsilon ) > 0$ for any $\epsilon \in ( 0, 2 ]$, where we define
\[
\delta_{X}(\epsilon)=\inf\left\{1-\frac{1}{2}\|x+\widetilde{x}\|_{X}\,\middle|\, x, \widetilde{x}\in X, \|x\|_{X}=\|\widetilde{x}\|_{X}=1, \|x-\widetilde{x}\|_{X}\geq\epsilon\right\} ; 
\] 
\item strictly convex if $\|\frac{1}{2}(x+\widetilde{x})\|_{X}<1$ for all $x, \widetilde{x} \in X$ with $\|x\|_{X}=\|\widetilde{x}\|_{X}=1$ and $x \neq \widetilde{x}$; 
\item smooth if for every $x\in X$ with $x\neq 0$ there is a unique $x^{\ast}\in X^{\ast}$ such that $\|x^{\ast}\|_{X^{\ast}}=1$ and $\langle x^{\ast}, x\rangle_{X^{\ast}, X}=\|x\|_{X}$.
\end{enumerate}
\end{definition}
\begin{remark}\label{remark:convex:smooth:s}
Every $p$-convex Banach space is uniformly convex and hence strictly convex. Every Hilbert space is $2$-convex with $C_{2}=1$, and its duality mapping~$J_{2}$ is the identity. Moreover, there is no nontrivial $p$-convex Banach space~$X$ for $p < 2$. 
\end{remark}

We recall some fundamental properties concerning the relationship between the geometry of Banach spaces and the duality mapping.
\begin{lemma}[\textup{\cite[Theorem~2.53]{SKHK}}] \label{thm:DM;SKHK}
The following statements hold true.
\begin{enumerate} \renewcommand{\labelenumi}{(\alph{enumi})} \setlength{\leftskip}{-0.15in}
\item For every $x\in X$, the set $J_{p}(x)$ is non-empty and convex.
\item For $p>1$, $X$ is smooth if and only if every duality mapping~$J_{p}$ is single-valued.
\item $X$ is strictly convex if and only if every duality mapping~$J_{p}$ is strictly monotone, that is, $\langle x^{\ast}-\widetilde{x}^{\ast},x-\widetilde{x}\rangle_{X^{\ast},X}>0$ for all $x,\widetilde{x}\in X$ with $x\neq \widetilde{x}$ and $x^{\ast}\in J_{p}(x)$, $\widetilde{x}^{\ast}\in J_{p}(\widetilde{x})$.
\end{enumerate}
\end{lemma}

Note that if $X$ is smooth, then $J_{p}$ is single-valued for $p>1$ by Lemma~\ref{thm:DM;SKHK}~(b). Hence, throughout this paper, whenever the underlying Banach space is smooth, we use $J_{p}$ itself instead of a single-valued selection~$j_{p}$.

\subsection{Bregman distances} \label{sec:preliminaries;BD}
Due to the geometric properties of Banach spaces, the Bregman distance provides a natural tool for the convergence analysis of regularization methods.

\begin{definition}[Bregman distance] \label{def:BD}
Let $X$ be a smooth Banach space and $p>1$. Then, the functional
\begin{equation} \label{BD}
D_{{p}}(x, \widetilde{x})=\frac{1}{p}\|x\|_{X}^{p}-\frac{1}{p}\|\widetilde{x}\|_{X}^{p}-\langle J_{p}(\widetilde{x}), x-\widetilde{x}\rangle_{X^{\ast},  X},\quad x, \widetilde{x}\in X,
\end{equation}
is called the Bregman distance with respect to the functional $\frac{1}{p}\|\cdot\|_{X}^{p}$.
\end{definition}
\begin{remark}
Since the Bregman distance is not symmetric in general, we note that the definition adopted in \cite{dHQS,MG} differs from Definition~\ref{def:BD} by interchanging the roles of $x$ and $\widetilde{x}$. On the other hand, \cite{Jin} adopts the same definition as Definition~\ref{def:BD}, although the Bregman distance is denoted by $\Delta_{p}(x,\widetilde{x})$. 
\end{remark}

It follows from \eqref{BD} that 
\begin{align} \label{Bregman;identity}
&D_{p}(w, u)-D_{p}(w, v) \nonumber \\
&=-D_{p}(u, v)+\langle J_{p}(u)-J_{p}(v), u-w\rangle_{X^{\ast},  X},\quad u, v, w\in X.
\end{align}

The following properties of the Bregman distance are well-known. See, for example, \cite[Theorem 2.60]{SKHK} and the references therein.
\begin{lemma} \label{thm:BD;SKHK}
Let $X$ be a smooth Banach space and $p>1$. Then, the following properties are valid:
\begin{enumerate} \renewcommand{\labelenumi}{(\alph{enumi})} \setlength{\leftskip}{-0.15in}
\item $D_{p}(x, \widetilde{x})\geq 0$ for all $x, \widetilde{x}\in X$.
\item If $X$ is also strictly convex, then $D_{p}(x, \widetilde{x})=0$ if and only if $x=\widetilde{x}$.
\item If $X$ is also uniformly convex, then $\| x_{n} - x \|_{X} \rightarrow 0$ if and only if $D_{p} ( x_{n} , x ) \rightarrow 0$. 
\item $X$ is $p$-convex if and only if $D_{p}(x, \widetilde{x})\geq\frac{C_{p}}{p}\|x-\widetilde{x}\|_{X}^{p}$ for all $x, \widetilde{x}\in X$, where $C_{p}$ is the constant in Definition~\ref{def:convex;smooth}~(a).
\end{enumerate}
\end{lemma}

\section{LM-method} \label{sec:preliminaries;LM}
In this section, we introduce the LM-method used in the subsequent analysis. Let $X, Y$ be real Banach spaces. We consider the nonlinear operator equation~\eqref{NOE}. Here, $x\in X$ is the unknown parameter and $y\in Y$ is the measurement data, and $F$ is the measurement operator. Let $x^{\dagger}$ generate the exact data~$y$, that is, $F(x^{\dagger})=y$. We also consider noisy data~$y^{\delta}$ satisfying
\begin{equation} \label{noiselevel}
\|y^{\delta}-y\|_{Y}\leq\delta.
\end{equation}

Starting with an initial guess $x_{0}^{\delta}=x_{0}$, the LM-method constructs the iterates as follows. Given the current iterate $x_{k}^{\delta}$, for each $\alpha>0$, consider the functional
\begin{equation} \label{prepLM}
\Phi_{\alpha}(x)=\frac{1}{r}\|y^{\delta}-F(x_{k}^{\delta})-F'(x_{k}^{\delta})(x-x_{k}^{\delta})\|_{Y}^{r}+\alpha D_{p}(x, x_{k}^{\delta}).
\end{equation}
We denote a minimizer of $\Phi_{\alpha}$ by $x_{k}(\alpha,y^{\delta})$. Then, $x_{k}(\alpha,y^{\delta})$ satisfies
\begin{align} \label{LM}
&\alpha(J_{p}(x_{k}(\alpha, y^{\delta}))-J_{p}(x_{k}^{\delta})) \nonumber \\
&\quad=F'(x_{k}^{\delta})^{\ast}J_{r}(y^{\delta}-F(x_{k}^{\delta})-F'(x_{k}^{\delta})(x_{k}(\alpha, y^{\delta})-x_{k}^{\delta})),
\end{align}
where $J_{r}: Y\to Y^{\ast}$ is the duality mapping. Here, \eqref{LM} corresponds to the LM-method for \eqref{NOE} in Banach spaces. For each $\alpha>0$, the minimizer $x_{k}(\alpha,y^{\delta})$ is uniquely determined and satisfies \eqref{LM}, see Lemma~\ref{lem:Jin31:alpha}. For some $0<\mu<1$, we define $\alpha_{k}(y^{\delta})>0$ to be the root of the equation
\begin{equation} \label{mu}
\left\|y^{\delta}-F(x_{k}^{\delta})-F'(x_{k}^{\delta})(x_{k}(\alpha,y^{\delta})-x_{k}^{\delta})\right\|_{Y}=\mu\left\|y^{\delta}-F(x_{k}^{\delta})\right\|_{Y}
\end{equation}
and write $\alpha_{k}=\alpha_{k}(y^{\delta})$ and $\psi_{k}^{\delta}(\alpha)=\left\|y^{\delta}-F(x_{k}^{\delta})-F'(x_{k}^{\delta})(x_{k}(\alpha,y^{\delta})-x_{k}^{\delta})\right\|_{Y}$. Once $\alpha_{k}>0$ is determined, we set $x_{k+1}^{\delta}=x_{k}(\alpha_{k}, y^{\delta})$. The well-definedness of $\alpha_{k}$ under \eqref{F;inv-Holder} follows from Corollay~\ref{cor:Jin33:alpha:noisy}. For noisy data, the iteration is terminated by Morozov's discrepancy principle at $k_{\ast}=k_{\ast}(\delta, y^{\delta})$ satisfying
\begin{equation} \label{MDP}
\|y^{\delta}-F(x_{k_{\ast}}^{\delta})\|_{Y}\leq\tau\delta<\|y^{\delta}-F(x_{k}^{\delta})\|_{Y},\quad 0\leq k<k_{\ast},
\end{equation}
where $\tau>1$. For exact data, we omit the superscript~$\delta$ and ignore the discrepancy principle. 
When $x_{k}$ has been defined, we define $\alpha_{k}$ by \eqref{mu}, and define $x_{k+1}$ by \eqref{LM} unless $F( x_{k} ) = y$. On the other hand, if $F(x_{k})=y$, then the H\"{o}lder stability estimate~\eqref{F;inv-Holder} and Lemma~\ref{thm:BD;SKHK}~(b) imply that $x_{k}=x^{\dagger}$, that is, the solution has been reached, and we set $x_{k+1}=x_{k}$. 

\subsection{Assumptions} \label{section:preliminaries;prep}
We make the following assumptions on the spaces~$X, Y$, the operator~$F$ and the initial guess~$x_{0}$.
\begin{assumption} \label{asm:XY} \mbox{} 
\begin{enumerate} \renewcommand{\labelenumi}{(\alph{enumi})} \setlength{\leftskip}{-0.1in}
\item $X$ is a nontrivial $p$-convex and smooth real Banach space.
\item $Y$ is a smooth real Banach space. 
\end{enumerate}
\end{assumption}
We note that $p \geq 2$ holds under Assumption~\ref{asm:XY} (a) by Remark~\ref{remark:convex:smooth:s}. 
\begin{assumption} \label{asm:op}
Let $B_{\rho}(x^{\dagger},D_{p})=\{x\in X \mid D_{p}(x^{\dagger}, x)\leq\rho\}\subset\mathcal{D}(F)$.
\begin{enumerate} \renewcommand{\labelenumi}{(\alph{enumi})} \setlength{\leftskip}{-0.1in}
\item The Fr\'{e}chet derivative $F'(\cdot)$ is Lipschitz continuous on $B_{\rho}(x^{\dagger},D_{p})$ such that
\begin{equation} \label{F'-Lip}
\|F'(x)-F'(\widetilde{x})\|_{\mathcal{L}(X, Y)}\leq L\|x-\widetilde{x}\|_{X},\quad x, \widetilde{x}\in B_{\rho}(x^{\dagger},D_{p})
\end{equation}
and
\begin{equation} \label{F'-bounded}
\|F'(x)\|_{\mathcal{L}(X, Y)}\leq\widehat{L},\quad x\in B_{\rho}(x^{\dagger},D_{p}),
\end{equation}
where $\|\cdot\|_{\mathcal{L}(X,Y)}$ denotes the operator norm for bounded operators from $X$ to $Y$, with a Lipschitz constant $L>0$ and a constant $\widehat{L}>0$.
\item Equation~\eqref{NOE} has the uniform H\"{o}lder stability estimate given as
\begin{equation} \label{F;inv-Holder}
D_{p}(x, \widetilde{x})\leq C_{F}^{p}\|F(x)-F(\widetilde{x})\|_{Y}^{\frac{1+\varepsilon}{2}p},\quad x, \widetilde{x}\in B_{\rho}(x^{\dagger},D_{p})
\end{equation}
for some $C_{F}>0$ and $\varepsilon\in(0, 1]$.
\end{enumerate}
\end{assumption}
\begin{assumption} \label{asm:x0} \mbox{}
\begin{enumerate} \renewcommand{\labelenumi}{(\alph{enumi})} \setlength{\leftskip}{-0.1in}
\item There exists a solution~$x^{\dagger}$ to \eqref{NOE}. 
\item The initial guess $x_{0}$ satisfies $x_{0}\in B_{\rho}(x^{\dagger},D_{p})$.
\end{enumerate}
\end{assumption}

\subsection{Preliminary estimates} \label{section:preliminaries;prep1}
We prepare several lemmas for the subsequent analysis. We first recall the following result on $x_{k}(\alpha,y^\delta)$. While the uniform smoothness of $X$ is assumed in \cite{Jin}, the smoothness of $X$ is sufficient for the result below.
\begin{lemma}[\textup{\cite[Lemma~3.1]{Jin}}] \label{lem:Jin31:alpha}
Let $X$ be a smooth and uniformly convex real Banach space and $p, r>1$. Suppose that Assumptions~\ref{asm:XY}~(b) and \ref{asm:op}~(a) hold. Then, for each $\alpha>0$, there exists a unique $x_{k}(\alpha, y^{\delta})$ satisfying \eqref{LM}. Moreover, the map $\alpha\mapsto x_{k}(\alpha,y^{\delta})$ is continuous on $(0, \infty)$, and $\psi_{k}^{\delta}$ is continuous and nondecreasing on $(0,\infty)$.
\end{lemma}

We next provide a criterion for the existence and uniqueness of $\alpha_{k}$ satisfying \eqref{mu}. In Hilbert spaces, as in \cite{KNS}, for example, the existence of~$\alpha_{k}$ is guaranteed by an explicit inequality whose proof relies on an orthogonal projection. Since such an argument is not directly available in Banach spaces, we instead formulate the following criterion based on the idea of \cite[Lemma~3.3]{Jin}.

\begin{lemma} \label{lem:general:alpha}
Let $X$ be a smooth and uniformly convex real Banach space and $p, r>1$. Suppose that Assumptions~\ref{asm:XY}~(b) and \ref{asm:op}~(a) hold. Fix $0\leq k<k_{\ast}$, and assume that $x_{k}^{\delta}$ is well-defined, and that there exists a constant $\widehat{c}>0$, independent of $\alpha$, such that 
\begin{equation} \label{condition:alpha:1}
\psi_{k}^{\delta}(\alpha)\geq\mu\left\|y^{\delta}-F(x_{k}^{\delta})\right\|_{Y}
\end{equation}
implies 
\begin{equation} \label{eq:lem:Jin32:alpha:1}
\|y^{\delta}-F(x_{k}^{\delta})\|_{Y}^{r}\leq\widehat{c} \alpha D_{p}(x^{\dagger},x_{k}^{\delta}) .
\end{equation}
Then, there exists a unique $\alpha_{k}>0$ satisfying \eqref{mu}.
\end{lemma}
\begin{proof}
First, note that there exists a unique $x_{k}(\alpha, y^{\delta})$ by Lemma~\ref{lem:Jin31:alpha}. 
Since $x_{k}(\alpha,y^{\delta})$ is a minimizer, 
\begin{math}
\Phi_{\alpha} ( x_{k} ( \alpha , y^{\delta} ) ) 
\leq \Phi_{\alpha} ( x_{k}^{\delta} ) 
\end{math}
holds. 
Thus, Lemma~\ref{thm:BD;SKHK}~(a) yields
\[
0\leq D_{p}(x_{k}(\alpha,y^{\delta}),x_{k}^{\delta})\leq\frac{1}{\alpha r}\|y^{\delta}-F(x_{k}^{\delta})\|_{Y}^{r}\to 0\qquad (\alpha\to\infty).
\]
Hence, Lemma~\ref{thm:BD;SKHK}~(c) gives $x_{k}(\alpha,y^{\delta})\to x_{k}^{\delta}$, and therefore
\begin{equation} \label{prep1forlem:Jin33:alpha:noisy}
\psi_{k}^{\delta}(\alpha)\to\|y^{\delta}-F(x_{k}^{\delta})\|_{Y}\qquad (\alpha\to\infty).
\end{equation}
We note that, since $k<k_{\ast}$, \eqref{MDP} gives
\begin{equation} \label{prep2forlem:Jin33:alpha:noisy}
\|y^{\delta}-F(x_{k}^{\delta})\|_{Y}>\tau\delta>0.
\end{equation}
Together with \eqref{prep2forlem:Jin33:alpha:noisy} and $0<\mu<1$, \eqref{prep1forlem:Jin33:alpha:noisy} shows that 
\eqref{condition:alpha:1} holds for sufficiently large $\alpha$. 
On the other hand, if \eqref{condition:alpha:1} holds for any $\alpha > 0$,
then
we obtain $y^{\delta} = F(x_{k}^{\delta})$
by taking the limit $\alpha \to 0$ in \eqref{eq:lem:Jin32:alpha:1}, 
which contradicts \eqref{prep2forlem:Jin33:alpha:noisy}. 
Hence, there exists $\widehat{\alpha} > 0$ satisfying 
$\psi_{k}^{\delta} ( \widehat{\alpha} ) < \mu \| y^{\delta} - F( x_{k}^{\delta} ) \|_{Y}$. 
Since $\psi_{k}^{\delta}$ is continuous 
by Lemma~\ref{lem:Jin31:alpha}, 
the intermediate value theorem implies the existence of $\alpha_{k}$. 

We next show the uniqueness of $\alpha_{k}$. Suppose that $\alpha_{k}^{\ell}$, $\ell=1,2$, are distinct positive roots of \eqref{mu}. Since $\psi_{k}^{\delta}(\alpha_{k}^{1})=\psi_{k}^{\delta}(\alpha_{k}^{2})$, the minimizing property of $x_{k}(\alpha_{k}^{\ell},y^{\delta})$ yields $D_{p}(x_{k}(\alpha_{k}^{1},y^{\delta}),x_{k}^{\delta})=D_{p}(x_{k}(\alpha_{k}^{2},y^{\delta}),x_{k}^{\delta})$. Thus, $x_{k}(\alpha_{k}^{1},y^{\delta})$ is a minimizer of both $\Phi_{\alpha_{k}^{1}}$ and $\Phi_{\alpha_{k}^{2}}$. 
Hence, for $\ell = 1, 2$, $\alpha_{k}^{\ell} ( J_{p} ( x_{k} ( \alpha_{k}^{1} , y^{\delta} ) ) - J_{p} ( x_{k}^{\delta} ) )$ equals the right-hand side of \eqref{LM} with $\alpha = \alpha_{k}^{1}$. 
Using $\alpha_{k}^{1}\neq\alpha_{k}^{2}$, we obtain $J_{p}(x_{k}(\alpha_{k}^{1},y^{\delta}))-J_{p}(x_{k}^{\delta})=0$. Since $X$ is strictly convex, Lemma~\ref{thm:DM;SKHK}~(c) implies that $J_{p}$ is strictly monotone. Hence, $x_{k}(\alpha_{k}^{1},y^{\delta})=x_{k}^{\delta}$. Consequently, \eqref{mu} contradicts \eqref{prep2forlem:Jin33:alpha:noisy} since $0<\mu<1$.
\end{proof}

We now derive an identity that will be useful throughout the analysis.
\begin{lemma} \label{lem:identity}
Let $X$ be a smooth real Banach space, and $p, r>1$. 
Suppose that Assumption~\ref{asm:XY}~(b) holds. 
Fix $0\leq k<k_{\ast}$,  $\alpha > 0$, and assume that $x_{k}(\alpha, y^{\delta})$ is well-defined. Then, we have
\begin{equation} \label{eq:lem:identity}
D_{p}(x^{\dagger}, x_{k}(\alpha, y^{\delta}))-D_{p}(x^{\dagger}, x_{k}^{\delta})=-D_{p}(x_{k}(\alpha, y^{\delta}), x_{k}^{\delta})-\frac{1}{\alpha}\psi_{k}^{\delta}(\alpha)^{r}+Z_{k}^{\delta},
\end{equation}
where
\[
Z_{k}^{\delta}=\frac{1}{\alpha}\langle J_{r}(y^{\delta}-F(x_{k}^{\delta})-F'(x_{k}^{\delta})(x_{k}(\alpha, y^{\delta})-x_{k}^{\delta})), y^{\delta}-F(x_{k}^{\delta})-F'(x_{k}^{\delta})(x^{\dagger}-x_{k}^{\delta})\rangle_{Y^{\ast},  Y}.
\]
\end{lemma}
\begin{proof}
Applying \eqref{Bregman;identity} with $w=x^{\dagger}$, $u=x_{k}^{\delta}$ and $v=x_{k}(\alpha, y^{\delta})$, and using \eqref{LM} yield
\begin{align} \label{prep1forlem:identity}
&D_{p}(x^{\dagger}, x_{k}(\alpha, y^{\delta}))-D_{p}(x^{\dagger},x _{k}^{\delta}) \nonumber\\
&=D_{p}(x_{k}^{\delta},x_{k}(\alpha, y^{\delta}))+Z_{k}^{\delta} \nonumber \\
&\quad-\frac{1}{\alpha}\left\langle J_{r}\left(y^{\delta}-F(x_{k}^{\delta})-F'(x_{k}^{\delta})(x_{k}(\alpha, y^{\delta})-x_{k}^{\delta})\right), y^{\delta}-F(x_{k}^{\delta})\right\rangle_{Y^{\ast}, Y}.
\end{align}
It follows from \eqref{LM} that 
\begin{align} \label{prep2forlem:identity}
&\left\langle J_{r}\left(y^{\delta}-F(x_{k}^{\delta})-F'(x_{k}^{\delta})(x_{k}(\alpha, y^{\delta})-x_{k}^{\delta})\right), y^{\delta}-F(x_{k}^{\delta})\right\rangle_{Y^{\ast}, Y} \nonumber \\
&=\psi_{k}^{\delta}(\alpha)^{r}+\alpha\left\langle J_{p}(x_{k}(\alpha, y^{\delta}))-J_{p}(x_{k}^{\delta}), x_{k}(\alpha, y^{\delta})-x_{k}^{\delta}\right\rangle_{X^{\ast}, X}.
\end{align}
By \eqref{Bregman;identity} with $u=x_{k}(\alpha, y^{\delta})$ and $v=w=x_{k}^{\delta}$, we have
\begin{align} \label{JpDp}
&\left\langle J_{p}(x_{k}(\alpha, y^{\delta}))-J_{p}(x_{k}^{\delta}), x_{k}(\alpha, y^{\delta})-x_{k}^{\delta}\right\rangle_{X^{\ast}, X} \nonumber \\
&=D_{p}(x_{k}(\alpha, y^{\delta}),x_{k}^{\delta})+D_{p}(x_{k}^{\delta},x_{k}(\alpha, y^{\delta})).
\end{align}
Combining \eqref{prep1forlem:identity}--\eqref{JpDp}, we obtain \eqref{eq:lem:identity}.
\end{proof}

We next establish a key estimate, which shows that a candidate in the next step of the LM-method is closer to $x^{\dagger}$ than the current iterate in terms of the Bregman distance.
\begin{lemma} \label{lem:Jin32:noisy}
Suppose that Assumptions \ref{asm:XY}, \ref{asm:op}, \ref{asm:x0}~(a) hold. Let $r>1$.  Assume further that
\begin{equation} \label{condition:Jin32:alpha}
C_{\mathrm{n}}=1-\frac{1}{\mu}\left[\frac{1}{\tau}+\left(1+\frac{1}{\tau}\right)\widetilde{C}_{\rho}\right]>0,
\end{equation}
where
$\widetilde{C}_{\rho} = \widetilde{C} \rho^{\frac{2\varepsilon}{(1+\varepsilon)p}}$ and $\widetilde{C} = 2^{-1} p^{\frac{2}{p}} C_{p}^{-\frac{2}{p}} LC_{F}^{\frac{2}{1+\varepsilon}}$. 
Fix $0\leq k<k_{\ast}$, and assume that $x_{k}^{\delta}\in B_{\rho}(x^{\dagger},D_{p})$ is well-defined. Then, for any $\alpha>0$ satisfying \eqref{condition:alpha:1}, we have \begin{equation} \label{eq:lem:Jin32:alpha}
\frac{1}{\alpha}\|y^{\delta}-F(x_{k}^{\delta})\|_{Y}^{r}\leq\frac{1}{C_{\mathrm{n}}\mu^{r}}\left(D_{p}(x^{\dagger},x_{k}^{\delta})-D_{p}(x^{\dagger},x_{k}(\alpha,y^{\delta}))\right).
\end{equation}
\end{lemma}
\begin{proof}
Recall $Z_{k}^{\delta}$ from Lemma~\ref{lem:identity}. Then, we have
\begin{equation} \label{prep1forlem:Jin32:alpha}
Z_{k}^{\delta}\leq\frac{1}{\alpha}\psi_{k}^{\delta}(\alpha)^{r-1}\left\|y^{\delta}-F(x_{k}^{\delta})-F'(x_{k}^{\delta})(x^{\dagger}-x_{k}^{\delta})\right\|_{Y}. 
\end{equation}
By the fundamental theorem of calculus for the Fr\'{e}chet derivative~\cite[Lemma\,A.63]{Kirsch} and \eqref{F'-Lip}, we have
\begin{equation} \label{FTCTD}
\left\|y-F(x_{k}^{\delta})-F'(x_{k}^{\delta})(x^{\dagger}-x_{k}^{\delta})\right\|_{Y}\leq\frac{L}{2}\left\|x^{\dagger}-x_{k}^{\delta}\right\|_{X}^{2}.
\end{equation}
Moreover, since $x^{\dagger}, x_{k}^{\delta}\in B_{\rho}(x^{\dagger},D_{p})$, it follows from Lemma~\ref{thm:BD;SKHK}~(d) and \eqref{F;inv-Holder} that
\begin{equation} \label{BregmanHolder}
\frac{C_{p}}{p}\|x^{\dagger}-x_{k}^{\delta}\|_{X}^{p}\leq D_{p}(x^{\dagger},x_{k}^{\delta})\leq C_{F}^{p}\left\|y-F(x_{k}^{\delta})\right\|_{Y}^{\frac{1+\varepsilon}{2}p}.
\end{equation}
Using \eqref{BregmanHolder} and $x_{k}^{\delta}\in B_{\rho}(x^{\dagger},D_{p})$, we have 
\begin{align} \label{prep2forlem:Jin32:alpha}
\|x^{\dagger}-x_{k}^{\delta}\|_{X}^{2}
& \leq \left( \frac{p}{C_{p}} \right)^{\frac{2}{p}}
D_{p} ( x^{\dagger} , x_{k}^{\delta} )^{\frac{2 \varepsilon}{( 1 + \varepsilon ) p}}
D_{p} ( x^{\dagger} , x_{k}^{\delta} )^{\frac{2}{( 1 + \varepsilon ) p}}
\nonumber \\
& \leq \left( \frac{p}{C_{p}} \right)^{\frac{2}{p}}
\rho^{\frac{2\varepsilon}{(1+\varepsilon)p}}
C_{F}^{\frac{2}{1 + \varepsilon}} 
\| y - F( x_{k}^{\delta} ) \|_{Y} . 
\end{align}
Combining \eqref{noiselevel}, \eqref{FTCTD}, \eqref{prep2forlem:Jin32:alpha}, \eqref{MDP} and \eqref{condition:alpha:1}, we obtain
\begin{align} \label{prep3forlem:Jin32:alpha}
&\left\|y^{\delta}-F(x_{k}^{\delta})-F'(x_{k}^{\delta})(x^{\dagger}-x_{k}^{\delta})\right\|_{Y}\leq\delta+\frac{L}{2}\left\|x^{\dagger}-x_{k}^{\delta}\right\|_{X}^{2} \nonumber \\
&\leq\delta+\widetilde{C}_{\rho}\left(\delta+\|y^{\delta}-F(x_{k}^{\delta})\|_{Y}\right)<\left[\frac{1}{\tau}+\left(1+\frac{1}{\tau}\right)\widetilde{C}_{\rho}\right]\|y^{\delta}-F(x_{k}^{\delta})\|_{Y} \nonumber \\
&\leq\frac{1}{\mu}\left[\frac{1}{\tau}+\left(1+\frac{1}{\tau}\right)\widetilde{C}_{\rho}\right]\psi_{k}^{\delta}(\alpha).
\end{align}
From \eqref{eq:lem:identity}, \eqref{prep1forlem:Jin32:alpha}, \eqref{prep3forlem:Jin32:alpha} and \eqref{condition:alpha:1}, we obtain 
\begin{equation} \label{prep4forlem:Jin32:alpha}
D_{p}(x^{\dagger}, x_{k}(\alpha, y^{\delta}))-D_{p}(x^{\dagger}, x_{k}^{\delta})\leq-D_{p}(x_{k}(\alpha, y^{\delta}), x_{k}^{\delta})-\frac{C_{\mathrm{n}}\mu^{r}}{\alpha}\|y^{\delta}-F(x_{k}^{\delta})\|_{Y}^{r}.
\end{equation}
Finally, \eqref{eq:lem:Jin32:alpha} follows from \eqref{prep4forlem:Jin32:alpha} and Lemma~\ref{thm:BD;SKHK}~(a).
\end{proof}

As an immediate consequence, we obtain the following estimate for exact data.
\begin{corollary} \label{cor:Jin32:exact}
Suppose that Assumptions \ref{asm:XY}, \ref{asm:op}, \ref{asm:x0}~(a) hold. Let $r>1$. Assume further that
\begin{equation} \label{condition:Jin32:alpha:exact}
C_{\mathrm{e}}=1-\frac{\widetilde{C}_{\rho}}{\mu}>0.
\end{equation}
Fix $k\in\mathbb{N}_{0}=\mathbb{N}\cup\{0\}$, and assume that $x_{k}\in B_{\rho}(x^{\dagger},D_{p})$ is well-defined. Then, the assertion of Lemma~\ref{lem:Jin32:noisy} remains valid with the superscript $\delta$ omitted  and $C_{\mathrm{n}}$ replaced by $C_{\mathrm{e}}$.
\end{corollary} 

Note that \eqref{eq:lem:Jin32:alpha} yields \eqref{eq:lem:Jin32:alpha:1} with $\widehat{c}=\frac{1}{C_{\mathrm{n}}\mu^{r}}$ by Lemma~\ref{thm:BD;SKHK}~(a). Hence, if $x_{k}^{\delta}\in B_{\rho}(x^{\dagger},D_{p})$, the following corollary follows from Lemmas~\ref{lem:Jin32:noisy} and~\ref{lem:general:alpha}.
\begin{corollary} \label{cor:Jin33:alpha:noisy}
Under the assumptions of Lemma~\ref{lem:Jin32:noisy}, there exists a unique $\alpha_{k}>0$ satisfying \eqref{mu}. Consequently, $x_{k+1}^{\delta} = x_{k} ( \alpha_{k} , y^{\delta} )$ is well-defined.
\end{corollary}

For exact data, the same result holds with Corollary~\ref{cor:Jin32:exact} in place of Lemma~\ref{lem:Jin32:noisy}.
\begin{remark} \label{rmk:welldefined}
To establish the well-definedness of the LM-method, it remains to show that $x_{k+1}^{\delta}\in B_{\rho}(x^{\dagger},D_{p})$, which is proved in Sections~\ref{section:exact}~and~\ref{section:noisy}. 
\end{remark}

Finally, we derive an estimate for $\alpha_{k}$. Henceforth, we write the corresponding residual explicitly instead of using $\psi_{k}^{\delta}$.
\begin{lemma} \label{lem:alphak}
Suppose that Assumption \ref{asm:XY} holds. Let $r>1$. Fix $0\leq k<k_{\ast}$, and assume that $x_{k}^{\delta}$ and $\alpha_{k}$ are well-defined. Then, we have
\begin{equation} \label{eq:lem:alphak}
\alpha_{k}\leq \frac{p}{2C_{p}}\frac{\mu^{r-1}}{(1-\mu)^{p-1}}\|F'(x_{k}^{\delta})\|_{\mathcal{L}(X,Y)}^{p}\|y^{\delta}-F(x_{k}^{\delta})\|_{Y}^{r-p}.
\end{equation}
\end{lemma}
\begin{proof}
Using \eqref{LM} yields
\begin{align} \label{prep1forlem:alphak}
&\alpha_{k}\left\langle J_{p}(x_{k+1}^{\delta})-J_{p}(x_{k}^{\delta}), x_{k+1}^{\delta}-x_{k}^{\delta}\right\rangle_{X^{\ast}, X} \nonumber \\
&=\left\langle F'(x_{k}^{\delta})^{\ast}J_{r}\left(y^{\delta}-F(x_{k}^{\delta})-F'(x_{k}^{\delta})(x_{k+1}^{\delta}-x_{k}^{\delta})\right), x_{k+1}^{\delta}-x_{k}^{\delta}\right\rangle_{X^{\ast}, X}.
\end{align}
It follows from \eqref{JpDp} and Lemma~\ref{thm:BD;SKHK}~(d) that
\begin{equation} \label{prep2forlem:alphak}
\left\langle J_{p}(x_{k+1}^{\delta})-J_{p}(x_{k}^{\delta}), x_{k+1}^{\delta}-x_{k}^{\delta}\right\rangle_{X^{\ast}, X}\geq\frac{2C_{p}}{p}\|x_{k+1}^{\delta}-x_{k}^{\delta}\|_{X}^{p}.
\end{equation}
Moreover, \eqref{mu} gives
\begin{align} \label{prep3forlem:alphak}
&\left\langle F'(x_{k}^{\delta})^{\ast}J_{r}\left(y^{\delta}-F(x_{k}^{\delta})-F'(x_{k}^{\delta})(x_{k+1}^{\delta}-x_{k}^{\delta})\right), x_{k+1}^{\delta}-x_{k}^{\delta}\right\rangle_{X^{\ast}, X} \nonumber \\
&\leq\mu^{r-1}\|y^{\delta}-F(x_{k}^{\delta})\|_{Y}^{r-1}\|F'(x_{k}^{\delta})\|_{\mathcal{L}(X,Y)}\|x_{k+1}^{\delta}-x_{k}^{\delta}\|_{X}.
\end{align}
On the other hand, by \eqref{mu}, we have
\begin{align*}
&\|F'(x_{k}^{\delta})\|_{\mathcal{L}(X,Y)}\|x_{k+1}^{\delta}-x_{k}^{\delta}\|_{X}\geq\|F'(x_{k}^{\delta})(x_{k+1}^{\delta}-x_{k}^{\delta})\|_{Y} \\
&\geq\left\|y^{\delta}-F(x_{k}^{\delta})\right\|_{Y}-\left\|y^{\delta}-F(x_{k}^{\delta})-F'(x_{k}^{\delta})(x_{k+1}^{\delta}-x_{k}^{\delta})\right\|_{Y} \\
&=(1-\mu)\|y^{\delta}-F(x_{k}^{\delta})\|_{Y} > 0.
\end{align*}
Note that \eqref{prep2forlem:Jin33:alpha:noisy} holds since $k<k_{\ast}$. Thus, $\|F'(x_{k}^{\delta})\|_{\mathcal{L}(X,Y)}>0$, and we obtain
\begin{equation} \label{prep4forlem:alphak}
\|x_{k+1}^{\delta}-x_{k}^{\delta}\|_{X}\geq(1-\mu)\frac{\|y^{\delta}-F(x_{k}^{\delta})\|_{Y}}{\|F'(x_{k}^{\delta})\|_{\mathcal{L}(X,Y)}}.
\end{equation}
Combining \eqref{prep1forlem:alphak}--\eqref{prep4forlem:alphak}, we obtain \eqref{eq:lem:alphak}.
\end{proof}

\section{Convergence Analysis for Exact Data} \label{section:exact} 
In this section, we establish the well-definedness and convergence of the LM-method for exact data and derive its convergence rates. Throughout this section, when referring to the corresponding formulas in Section~\ref{sec:preliminaries;LM}, we omit the superscript~$\delta$.

\begin{theorem} \label{thm:exact}
Suppose that Assumptions~\ref{asm:XY}, \ref{asm:op}, \ref{asm:x0} and \eqref{condition:Jin32:alpha:exact} hold. Let $r>1$. Then, $x_{k}$ is well-defined for all $k\in\mathbb{N}_{0}$, and $x_{k}\to x^{\dagger}$ as $k\to\infty$. Moreover, with 
\[
c=\frac{2C_{\mathrm{e}}C_{p}\mu(1-\mu)^{p-1}}{p\widehat{L}^{p}C_{F}^{\frac{2p}{1+\varepsilon}}},
\]
the convergence rates are given by
\begin{numcases}{D_{p}(x^{\dagger},x_{k})\leq}
\rho(1-c)^{k} & for $\varepsilon=1$ \label{eq:thm:exact:rate1} \\
\left(ck\frac{1-\varepsilon}{1+\varepsilon}+\rho^{-\frac{1-\varepsilon}{1+\varepsilon}}\right)^{-\frac{1+\varepsilon}{1-\varepsilon}}& for $\varepsilon\in(0, 1)$ \label{eq:thm:exact:rate2}
\end{numcases}
\end{theorem}
\begin{remark}
If $F(x_{k})=y$ for some $k\in \mathbb{N}_{0}$, then $x_{j}=x^{\dagger}$ for all $j\geq k$. Hence, the conclusions of Theorem~\ref{thm:exact} hold in this case. Thus, in what follows, we assume that $F(x_{k})\neq y$ for all $k\in\mathbb{N}_{0}$.
\end{remark}
\begin{proof}[Proof of Theorem~\ref{thm:exact}]
We first show inductively that $x_{k}\in B_{\rho}(x^{\dagger},D_{p})$ is well-defined for all $k\in\mathbb{N}_{0}$. This holds for $k=0$ by Assumption~\ref{asm:x0}~(b). Suppose that $x_{k}\in B_{\rho}(x^{\dagger},D_{p})$ is well-defined for some $k\in\mathbb{N}_{0}$. Corollary~\ref{cor:Jin33:alpha:noisy} for exact data implies that $\alpha_{k}$ is well-defined and hence so is $x_{k+1}$. 
Since $x_{k}\in B_{\rho}(x^{\dagger},D_{p})$ holds and \eqref{mu} is satisfied with equality by the definition of $\alpha_{k}$, Corollary~\ref{cor:Jin32:exact}, together with \eqref{eq:lem:alphak} and \eqref{F'-bounded}, 
\begin{equation} \label{prep1forthm:exact}
D_{p}(x^{\dagger},x_{k+1})-D_{p}(x^{\dagger},x_{k})\leq-\frac{2C_{\mathrm{e}}C_{p}\mu(1-\mu)^{p-1}}{p\widehat{L}^{p}}\|y-F(x_{k})\|_{Y}^{p}\leq0.
\end{equation}
In particular, $D_{p}(x^{\dagger},x_{k+1})\leq D_{p}(x^{\dagger},x_{k})\leq\rho$ and hence $x_{k+1}\in B_{\rho}(x^{\dagger},D_{p})$. This completes the induction as noted in Remark~\ref{rmk:welldefined}. Therefore, $x_{k}$ is well-defined and belongs to $B_{\rho}(x^{\dagger},D_{p})$ for all $k\in\mathbb{N}_{0}$. By \eqref{prep1forthm:exact} and Lemma~\ref{thm:BD;SKHK}~(a), $D_{p}(x^{\dagger},x_{k})$ is nonincreasing and nonnegative. Hence, \eqref{prep1forthm:exact} implies $\|y-F(x_{k})\|_{Y}\to 0$. Then, together with \eqref{F;inv-Holder}, Lemma~\ref{thm:BD;SKHK}~(c) yields $x_{k}\to x^{\dagger}$.

Moreover, it follows from \eqref{F;inv-Holder} and \eqref{prep1forthm:exact} that
\begin{equation} \label{prep2forthm:exact}
D_{p}(x^{\dagger},x_{k+1})-D_{p}(x^{\dagger},x_{k})\leq-cD_{p}(x^{\dagger},x_{k})^{\frac{2}{1+\varepsilon}}.
\end{equation}
For $\varepsilon=1$, \eqref{prep2forthm:exact} and Assumption~\ref{asm:x0}~(b) yield \eqref{eq:thm:exact:rate1}. For $\varepsilon\in(0,1)$, since
\[
D_{p}(x^{\dagger},x_{k+1})\leq D_{p}(x^{\dagger},x_{k})(1-cD_{p}(x^{\dagger},x_{k})^{\frac{1-\varepsilon}{1+\varepsilon}})
\]
holds by \eqref{prep2forthm:exact}, using $(1-s)^{-\beta}\geq1+\beta s$ for $0\leq s<1$ and $0\leq\beta\leq 1$ implies
\begin{equation} \label{prep3forthm:exact}
D_{p}(x^{\dagger}, x_{k+1})^{-\frac{1-\varepsilon}{1+\varepsilon}}\geq D_{p}(x^{\dagger}, x_{k})^{-\frac{1-\varepsilon}{1+\varepsilon}}+c\frac{1-\varepsilon}{1+\varepsilon}.
\end{equation}
Combining \eqref{prep3forthm:exact} with Assumption~\ref{asm:x0}~(b), we obtain \eqref{eq:thm:exact:rate2}. 
\end{proof}

\section{Convergence Analysis for Noisy Data} \label{section:noisy}
In this section, we establish the well-definedness and monotonicity properties of the LM-method for noisy data and derive an error estimate and an estimate for the stopping index $k_{\ast}$.
\begin{theorem} \label{thm:noisy}
Suppose that Assumptions~\ref{asm:XY}, \ref{asm:op}, \ref{asm:x0} and \eqref{condition:Jin32:alpha} hold. Let $r>1$. Then, $x_{k}^{\delta}$ is well-defined for $0\leq k\leq k_{\ast}$, and $k_{\ast}$ is finite. Moreover,
\begin{equation} \label{noisy:mono}
D_{p}(x^{\dagger},x_{k+1}^{\delta})\leq D_{p}(x^{\dagger},x_{k}^{\delta})
\end{equation}
for $0\leq k<k_{\ast}$ and
\begin{equation} \label{eq:thm:noisy:rate:s}
D_{p}(x^{\dagger},x_{k_{\ast}}^\delta)\leq \rho - R \tau^{p} k_{\ast} \delta^{p}
\end{equation}
holds, where $R = 2 C_{\mathrm{n}}C_{p}\mu(1-\mu)^{p-1} p^{-1} \widehat{L}^{-p}$. 
In addition, under 
\begin{equation} \label{eq:thm:noisy:kast}
\mu<1-2\widetilde{C}_{\rho},
\end{equation}
we have $k_{\ast}(\delta, y^{\delta})=O(1+|\log\delta|)$.
\end{theorem}
\begin{remark}
The estimate~\eqref{eq:thm:noisy:rate:s} shows that, starting from the initial bound~$\rho$, the upper bound on the Bregman distance is improved at the stopping index $k_{\ast}$. A result corresponding to that obtained for the Landweber iteration in \cite{MG} can also be derived. Indeed, for a given $\delta>0$, if $\rho\leq C\delta^{p}$ for some $C>0$, then \eqref{eq:thm:noisy:rate:s} yields $D_{p}(x^{\dagger},x_{k_{\ast}}^\delta)\leq C'\delta^{p}$, where $C' = C - R \tau^{p}k_{\ast}$. 
\end{remark}
\begin{proof}[Proof of Theorem~\ref{thm:noisy}]
By the same argument as in the proof of Theorem~\ref{thm:exact}, using Corollary~\ref{cor:Jin33:alpha:noisy}, it follows by induction that $x_{k}^{\delta}$ is well-defined and belongs to $B_{\rho}(x^{\dagger},D_{p})$ for $0\leq k\leq k_{\ast}$. In particular, \eqref{noisy:mono} holds for $0\leq k< k_{\ast}$.

Let us show that the LM-method terminates after finitely many iterations. Note that Lemma~\ref{lem:Jin32:noisy} and the definition of $\alpha_{k}$ imply that \eqref{eq:lem:Jin32:alpha} holds with $\alpha=\alpha_k$. It follows from \eqref{eq:lem:Jin32:alpha} with $\alpha=\alpha_{k}$, \eqref{eq:lem:alphak} and \eqref{F'-bounded} that
\begin{equation} \label{prep1forthm:noisy}
R \|y^{\delta}-F(x_{k}^{\delta})\|_{Y}^{p}\leq D_{p}(x^{\dagger},x_{k}^{\delta})-D_{p}(x^{\dagger},x_{k+1}^{\delta}).
\end{equation}
Combining \eqref{MDP}, \eqref{prep1forthm:noisy}, Assumption~\ref{asm:x0}~(b) and Lemma~\ref{thm:BD;SKHK}~(a), we obtain
\begin{align*} 
R (\tau\delta)^{p}k_{\ast}
& <R\sum_{k=0}^{k_{\ast}-1}\|y^{\delta}-F(x_{k}^{\delta})\|_{Y}^{p} \\
& \leq D_{p}(x^{\dagger},x_{0})-D_{p}(x^{\dagger},x_{k_{\ast}}^{\delta})
\leq\rho-D_{p}(x^{\dagger},x_{k_{\ast}}^{\delta}) \leq \rho 
\end{align*}
and hence we obtain 
$k_{\ast}<\rho R^{-1}(\tau\delta)^{-p}<\infty$ 
and \eqref{eq:thm:noisy:rate:s}.

Finally, we derive an estimate for $k_{\ast}$. From \eqref{prep4forlem:Jin32:alpha} and Lemma~\ref{thm:BD;SKHK}~(a), together with \eqref{noisy:mono} and Assumption~\ref{asm:x0}~(b), we obtain
\begin{equation} \label{prep3forthm:noisy}
D_{p}(x_{k+1}^{\delta},x_{k}^{\delta})\leq D_{p}(x^{\dagger},x_{k}^{\delta})\leq D_{p}(x^{\dagger},x_{0})\leq\rho.
\end{equation}
Since $x_{k}^{\delta}, x_{k+1}^{\delta}\in B_{\rho}(x^{\dagger}, D_{p})$, by \eqref{prep3forthm:noisy} and the same argument as in Lemma~\ref{lem:Jin32:noisy}, we have 
\begin{align} \label{prep4forthm:noisy}
&\left\|F(x_{k+1}^{\delta})-F(x_{k}^{\delta})-F'(x_{k}^{\delta})(x_{k+1}^{\delta}-x_{k}^{\delta})\right\|_{Y} \nonumber \\
&\leq \widetilde{C}_{\rho} \left\|F(x_{k+1}^{\delta})-F(x_{k}^{\delta})\right\|_{Y}
\leq\widetilde{C}_{\rho}\left(\left\|y^{\delta}-F(x_{k+1}^{\delta})\right\|_{Y}+\left\|y^{\delta}-F(x_{k}^{\delta})\right\|_{Y}\right).
\end{align}
It follows from \eqref{mu} and \eqref{prep4forthm:noisy} that
\begin{align} \label{prep5forthm:noisy}
&\mu\left\|y^{\delta}-F(x_{k}^{\delta})\right\|_{Y} \nonumber \\
&\geq\left\|y^{\delta}-F(x_{k+1}^{\delta})\right\|_{Y}-\left\|F(x_{k+1}^{\delta})-F(x_{k}^{\delta})-F'(x_{k}^{\delta})(x_{k+1}^{\delta}-x_{k}^{\delta})\right\|_{Y} \nonumber \\
&\geq(1-\widetilde{C}_{\rho})\left\|y^{\delta}-F(x_{k+1}^{\delta})\right\|_{Y}-\widetilde{C}_{\rho}\left\|y^{\delta}-F(x_{k}^{\delta})\right\|_{Y}.
\end{align}
Iterating \eqref{prep5forthm:noisy} and using \eqref{MDP} yield
\begin{equation} \label{prep6forthm:noisy}
\tau\delta<\left\|y^{\delta}-F(x_{k_{\ast}-1}^{\delta})\right\|_{Y}\leq\left(\frac{\mu+\widetilde{C}_{\rho}}{1-\widetilde{C}_{\rho}}\right)^{k_{\ast}-1}\left\|y^{\delta}-F(x_{0})\right\|_{Y}.
\end{equation}
By \eqref{eq:thm:noisy:kast}, $(\mu+\widetilde{C}_{\rho})/(1-\widetilde{C}_{\rho})<1$, while \eqref{noiselevel} gives $\left\|y^{\delta}-F(x_{0})\right\|_{Y}\leq\delta+\left\|y-F(x_{0})\right\|_{Y}$. Hence, taking logarithms in \eqref{prep6forthm:noisy} completes the proof.
\end{proof}

\bibliographystyle{abbrv}
\bibliography{ref}

\bigskip \noindent
Akari Ishida \\
Graduate School of Science, University of Hyogo, \\
2167 Shosha, Himeji, Hyogo 671-2201, Japan \\
Graduate School of Mathematics, Nagoya University, \\
Furocho, Chikusa-ku, Nagoya, Aichi 464-8602, Japan \\
E-mail: \texttt{akari.ishida@sci.u-hyogo.ac.jp} 

\bigskip \noindent
Sei Nagayasu \\
Graduate School of Science, University of Hyogo, \\
2167 Shosha, Himeji, Hyogo 671-2201, Japan \\
E-mail: \texttt{sei@sci.u-hyogo.ac.jp}
\end{document}